\documentclass[a4paper,12pt]{article}
\usepackage{amsmath}
\usepackage{amssymb}
\usepackage{mathrsfs}
\usepackage{amsfonts}
\usepackage{amsthm}
\usepackage{empheq}
\usepackage{pstricks, pst-node, pst-text, pst-3d,psfrag}
\usepackage{graphicx}
\newtheorem{theorem}{Theorem}[section]
\newtheorem{definition}[theorem]{Definition}
\newtheorem{lemma}[theorem]{Lemma}
\newtheorem{proposition}[theorem]{Proposition}
\newtheorem{rem}[theorem]{Remark}
\newtheorem{Cor}[theorem]{Corollary}
\numberwithin{equation}{section}

\makeatletter % `@' now normal "letter"
\@addtoreset{equation}{section}
\makeatother  % `@' is restored as "non-letter"

\newcommand\norm[1]{\lVert#1\rVert}
\newcommand\abs[1]{\lvert#1\rvert}

\begin{document}
\title{Transversality for Cyclic Negative Feedback Systems}

\author {
\\
Yi Wang\thanks{Partially supported by NSF of China No. 11371338 and 91130016.}
  $\,$ and Dun Zhou\\
School of Mathematical Science\\
 University of Science and Technology of China
\\ Hefei, Anhui, 230026, P. R. China
\\
}
\date{}

\maketitle
% insert the table of contents
%\tableofcontents

%---------------------SECTION DIVIDE LINE---------------------------
\begin{abstract}
 Transversality of stable and unstable manifolds of hyperbolic periodic trajectories is proved for monotone cyclic systems with negative feedback. Such systems in general are not in the category of monotone dynamical systems in the sense of Hirsch. Our main tool utilized in the proofs is the so-called cone of high rank. We further show that stable and unstable manifolds between a hyperbolic equilibrium and a hyperbolic periodic trajectory, or between two hyperbolic equilibria with different dimensional unstable manifolds also intersect transversely.
\end{abstract}

\section{Introduction}
Oscillations frequently occur and play a fundamental role
in biological systems and networks.
%Therefore, large quantities of investigation focus on
%the underling principle for a system to admit oscillations
%in a fluctuating environment.
It has been widely observed that many biological oscillators have a cyclic structure consisting of negative feedback loops.
Such cyclic nature of interactions
appears in neural systems, cellular control systems and the description of cascades of enzimatic reactions coupled with gene
transcription (see e.g., \cite{Tsai,Fer,Hasting1977,EL}). Typical examples of cyclic negative feedback models include
the Goodwin oscillator, a well-studied model relevant to circadian
oscillations (\cite{Goodw}); the Repressilator,
a transcriptional negative feedback loop constructed
in Escherichia coli (\cite{EL,MuHo}); the Metabolator, a synthetic
metabolic oscillator (\cite{Fung}); and the Frzilator, a model
of the control of gliding motions in myxobacteria (\cite{Igo}), etc.

Consequently,  negative feedbacks which are embedded in a cyclic architecture, are believed to be the
underling principle for a system to admit oscillations in a fluctuating environment.
For such classes of models, many results can be found in the literature (see  e.g., \cite{Hasting1977,EL,FrT,Ton}). In particular, all the oscillator models previously introduced can be written in
an abstract form as
\begin{equation}\label{feedback-equation}
\begin{split}
\dot{x}_1 &=f_1(x_1,x_n),\\
\dot{x}_i &=f_i(x_i,x_{i-1}),\quad 2\leq i\leq n-1,\\
\dot{x}_n &=f_n(x_{n},x_{n-1}).\\
%\tag{1}
\end{split}
\end{equation}
where the nonlinearity $f=(f_1,f_2,\cdots,f_n)$, together with their partial derivatives with respect to $x_j$, are continuous in $\mathbb{R}^n$ and that there exists $\delta_i\in \{-1,1\}$, $1\leq i\leq n$, such that
\begin{equation}\label{feedback-condition}
\delta_i\frac{\partial f_i(x_i,x_{i-1})}{\partial x_{i-1}}>0,\quad \text{for all}\ (x_i,x_{i-1})\in \mathbb{R}^2.
\end{equation}
A remarkable result has been accomplished by
Mallet-Paret and Smith \cite{Mallet-Paret1990}: They have shown that the omega-limit
set of any bounded orbit of system \eqref{feedback-equation}-\eqref{feedback-condition}
 can be embedded in $\mathbb{R}^2$, and hence, the Poincar\'{e}-Bendixson property severely constrains possible dynamics of the system. Such insight confirms that a cyclic structure consisting of negative feedback loops is responsible for the emergence of oscillations in biological systems.

Following \cite{Mallet-Paret1990}, we call system \eqref{feedback-equation}-\eqref{feedback-condition} a monotone cyclic feedback system (MCFS). Let $\Delta=\delta_1 \delta_2\cdots \delta_n$, then there are two
types of MCFS depending on the
sign of $\Delta$.  If $\Delta=1$ (resp. $\Delta=-1$), then
system \eqref{feedback-equation}-\eqref{feedback-condition} is called a MCFS with positive
(resp. negative) feedback. A MCFS with positive
feedback ($\Delta=1$) is in particular a monotone dynamical system in the sense of Hirsch \cite{HiSm,Smi95} with respect to
certain usual convex cone and many classical
results for monotone dynamical systems contained in \cite{HiSm,Smi95} apply to \eqref{feedback-equation}-\eqref{feedback-condition}.  However, if $\Delta=-1$, such system is not monotone in the usual sense of Hirsch \cite{HiSm,Smi95}.

 In the theory of dynamical systems, transversality of stable and unstable
manifolds of critical elements plays a central role in connection with structural
stability (see e.g., \cite{Palis}). Despite this fact, there are not many results in the literature to verify if
transversality holds for a given dynamical system. Fusco and Oliva \cite{Fusco1987,fusco1990} have presented two classes of finite-dimensional cooperative ODE systems which possess the transversality. For scalar parabolic equations, Henry \cite{Hen} and Angenent \cite{Ang} have proved transversality of the invariant manifolds of stationary solutions  (see also Chen et al. \cite{Chen} for time-periodic cases) with separated boundary condition. For periodic boundary condition,
Czaja and Rocha \cite{CRo} have recently shown that the stable and
unstable manifolds of two hyperbolic periodic orbits always intersect transversally. The other automatic
transversality results have been completed in \cite{JR1,JR2}.

Going back to the MCFS \eqref{feedback-equation}-\eqref{feedback-condition}. When the feedback is positive
(i.e., $\Delta=1$), the main results in Fusco and Oliva \cite{fusco1990} may imply that any connecting orbit between two hyperbolic periodic orbits or between a hyperbolic periodic orbit and a hyperbolic equilibrium is automatically transversal.

However, it deserves to point out that all the aforementioned systems, in both finite-dimensional and infinite-dimensional settings, fall in the category of monotone dynamical systems in the sense of Hirsch \cite{HiSm,Smi95}. To the best of our knowledge, there are very few nontrivial explicit examples outside the category of monotone dynamical systems, where invariant manifolds of critical elements (particularly, of periodic orbits) are known to intersect transversely.

In this paper, we will extensively focus on system \eqref{feedback-equation}-\eqref{feedback-condition} with negative feedback ($\Delta=-1$).  Our main purpose is to show that this system admits transversality of stable and unstable
manifolds of critical elements. As we mentioned before, such system is not monotone in the usual sense of Hirsch. So, we presented here a class of explicit systems, not in the category of Hirsch \cite{HiSm,Smi95} but including many cyclic negative feedback biological models, for which ``transversality" property holds.

Our approach is motivated by the recent work of Sanchez \cite{San1,San2} on a newly-extended notion of monotone flows with respect to certain so-called cones of rank $k$. These cones were already considered by Fusco and Oliva \cite{fusco1991}  (see also Krasnoselskij et al. \cite{Kra} for infinite-dimensional settings). Such cones consist of straight lines and contain a $k$-dimensional linear subspace and no higher dimensional subspace. A usual convex cone $K$ (in the sense of Hirsch \cite{HiSm})
defines the generalized cone $K\cup (-K)$ which is of rank $1$.
For system \eqref{feedback-equation}-\eqref{feedback-condition} with negative feedback, Mallet-Paret and Smith \cite{Mallet-Paret1990} introduced an integer-valued
Lyapunov functional $N$. This function $N$ is not defined everywhere but
only on an open and dense subset of $\mathbb{R}^n$ on which it is also continuous. It is locally constant near points where it is defined and strictly decreasing as $t$ increases through points where it is not defined. The existence of $N$ enables us to present two modified functions of $N$ (see Lemma \ref{zero-property2}) and construct a family of nested cones, say $K_1\subset K_2\cdots \subset K_j$, of even rank (except that the largest cone $K_j$ is of odd rank when $n$ is an odd number), and obtain monotonicity of the system  with respect to these high-rank cones (see Proposition \ref{map-int}).

 In particular, if system \eqref{feedback-equation}-\eqref{feedback-condition} is linear, by virtue of the generalized Perron-Frobenius Theorem with respect to high-rank cones (\cite[Theorem 1]{fusco1991}, see also \cite{Kra} for the infinite dimensional settings), we are able to decouple $\mathbb{R}^n$ into many $2$-dimensional invariant subspaces $W_1,W_2,\cdots,W_j$ (When $n$ is odd, the last space $W_j$ is just $1$-dimensional) of the corresponding solution operator. Moreover, the growth rates of the solution operator on different invariant subspaces are strictly separated (see Lemma \ref{rootspace-arra}). As a consequence, we here generalize the Floquet theory established in Mallet-Paret and Smith \cite{Mallet-Paret1990} for time-periodic cases to general time-dependent cases by appealing a different approach.

Based on the theory obtained above and motivated by \cite{Fusco1987,fusco1990}, we are able to investigate transversality of stable and unstable
manifolds of critical elements of the system. More precisely, we will show that for any two hyperbolic periodic
orbits $\Gamma^{-}$ and $\Gamma^{+}$, the unstable manifold $W^u(\Gamma^{-})$ of $\Gamma^{-}$ and the stable manifold $W^s(\Gamma^{+})$ of $\Gamma^{+}$ will always intersect transversely (see Theorem \ref{transversality}). Moreover, such ``automatic" transversality will still hold if one of the two periodic orbits ($\Gamma^{+}$ or $\Gamma^{-}$) is replaced by a hyperbolic equilibrium. When considering  transversality between two hyperbolic equilibria, we show that if the dimensions of their unstable manifolds are different, then their corresponding stable and unstable manifolds will also intersect transversely.

This paper is organized as follows. In section 2, we first collect some properties of the integer-valued
Lyapunov function $N$ introduced in \cite{Mallet-Paret1990}; and then present two modified functions of $N$,
by which one can define the nested cones of high-rank so that the flow generated by \eqref{feedback-equation}-\eqref{feedback-condition} with negative feedback is monotone with respect to these high-rank cones. Moreover, if system \eqref{feedback-equation}-\eqref{feedback-condition} is linear, we generalize the Floquet theory in \cite{Mallet-Paret1990} for time-periodic cases to general time-dependent cases by the generalized Perron-Frobenius Theorem for high-rank cones. In section 3, we proved transversality of the stable and unstable manifolds of critical elements for system \eqref{feedback-equation}-\eqref{feedback-condition}.

\section{Cones of high-rank in Linear System}
In this section, we will introduce and investigate cones of high-rank for the linear negative feedback system
\begin{equation}\label{linearity-com-system}
\begin{split}
\dot{x}_1 &=a_{11}(t)x_1+a_{1n}(t)x_n,\\
\dot{x}_i &=a_{i,i-1}(t)x_{i-1}+a_{ii}(t)x_{i},\quad 2\leq i\leq n-1,\\
\dot{x}_n &=a_{n,n-1}(t)x_{n-1}+a_{nn}(t)x_n,\\
%\tag{1}
\end{split}
\end{equation}
with all the coefficient functions being continuous on $\mathbb{R}$ and satisfying the following condition:
\begin{equation}\label{positive-a-a}
a_{1,n}(t)<0\, \textnormal{ and }\,a_{i,i-1}(t)>0,\quad i=2,\cdots,n-1,
\end{equation}
for all $t\in \mathbb{R}$. Combining with the generalized Perron-Frobenius Theorem developed by \cite{fusco1991}, we will eventually split $\mathbb{R}^n$ into many invariant subspaces, whose dimension is no more than $2$, of the solution operator of system \eqref{linearity-com-system}. Hereafter, we always write the coefficient matrix as $A(t)=(a_{ij}(t))_{n\times n}$.

We now introducing an integer-valued Lyapunov function $N$ associated with \eqref{linearity-com-system}. From \cite{Mallet-Paret1990}, if we denote the set $\Lambda=\{x|x\in \mathbb{R}^n \ \text{and}\ x_i\neq 0,i=1,2,\cdots,n\}$, then one can define a continuous map $N$ on $\Lambda$, taking values in $\{0,1,2,\cdots,n\}$, by
\[
N(x)=\text{card} \{i|\delta_i x_i x_{i-1}<0\},
\]
while here $\delta_1=-1$ and $\delta_i=1,\ 2\leq i\leq n$.
 Henceforth, we let $\tilde{n}=n$ for $n$ is odd and $\tilde{n}=n-1$ for $n$ is even. Moreover, it follows from \cite{Mallet-Paret1990} that $$N(x)\in\{1,3,\cdots,\tilde{n}\}$$  for any $x\in \Lambda$.

Clearly, $\Lambda$ is open and dense in $\mathbb{R}^n$. Motivated by \cite{Fusco1987,fusco1990}, we now define two functions
\begin{equation*}
N_{m},N_M:\mathbb{R}^n\rightarrow \{1,3,\cdots,\tilde{n}\}
\end{equation*}
by letting $N_{m}(x)$, $N_M(x)$ be the minimum and maximum value of $N(x')$ for $x'\in \mathcal{U}\cap \Lambda$, where $\mathcal{U}$ being a small neighborhood of $x\in \mathbb{R}^n$. These two functions will then help us extend (continuously) the domain $\Lambda$ of $N$ to
\[
\mathcal{N}=\{x\in \mathbb{R}^n|N_m(x)=N_M(x)\}.
\]
Note that $\mathcal{N}$ is also open and dense in $\mathbb{R}^n$ and $\mathcal{N}$ is the maximal domain on which $N$ is continuous.
\begin{lemma}\label{zero-property}
Let $x(t)$ be a nontrivial solution of \eqref{linearity-com-system}. Then:
\begin{itemize}
\item [\emph{(i)}] $x(t)\in \mathcal{N}$ except at isolated values of $t$ and $N(x(t))$ is nonincreasing as $t$ increases with $x(t)\in \mathcal{N}$;
\item [\emph{(ii)}]If $x(t_0)\notin \mathcal{N}$, then for $\varepsilon>0$ small, one has $N(x(t_0+\varepsilon))<N(x(t_0-\varepsilon))$;
\item [\emph{(iii)}] There exists a $t_0>0$ such that $x(t)\in \mathcal{N}$ and $N(x(t))$ is constant for $t\in [t_0,+\infty)$ and for $t\in (-\infty,-t_0],$ respectively.

\end{itemize}
\end{lemma}
\begin{proof}
See \cite[Proposition 1.1]{Mallet-Paret1990} for (i) and (ii). It follows from (i) and (ii) that $N(x(t))$ can drop to a lower value only finitely many times, which implies (iii).
\end{proof}
Moreover, we have the following additional property of the relation between $N$ and $N_m$ (resp. $N_M$).
\begin{lemma}\label{zero-property2}
Let $x(t)$ be a nontrivial solution of \eqref{linearity-com-system}.
If $x(t_0)\notin\mathcal{N}$, then for $\varepsilon>0$ small enough, one has
\begin{equation}\label{min-max}
N(x(t_0+\varepsilon))=N_m(x(t_0))\,\,
    \textnormal{ and }\,\,N(x(t_0-\varepsilon))=N_M(x(t_0)).
\end{equation}
\end{lemma}
Before proving this lemma, we need the following technical lemma.
\begin{lemma}\label{perturb}
Let $y(t)$ be the solution of
\begin{equation*}\label{perturb-equation}
  \dot{y}=B(t)y+g(t),\quad y(0)=0,
\end{equation*}
where  $B(t)$ is a continuous $n\times n$ matrix function and $g(t)$ is a continuous $n$-vector valued function satisfying
$g(t)=g_mt^m+o(t^m)$, as $t\to 0$. Here $g_m\in \mathbb{R}^n$ and $m$ is a nonnegative integer. Then, one has
\[
y(t)=\frac{g_m}{m+1}t^{m+1}+o(t^{m+1}),\,\, \textnormal{ as } t\to 0.
\]
\end{lemma}
\begin{proof}
This lemma is directly from the L'Hospital principle. (See also \cite[p.374]{Mallet-Paret1990}).
\end{proof}

\begin{proof}[Proof of Lemma \ref{zero-property2}]
Without loss of generality we assume that $t_0=0$. We first consider the case that solution $x(t)$ with initial value $x(0)=(x_1,x_2,\cdots,x_n)$, where $x_1\neq 0$ and $x_i=0$ for $i=2,\cdots,n$. For each $i=2,\cdots,n$, the equation
\[
\dot{x}_i(t)=a_{i,i-1}(t)x_{i-1}(t)+a_{i,i}(t)x_i(t)
\]
satisfies the assumptions in Lemma \ref{perturb}. Therefore,
$$x_2(t)=a_{2,1}(0)x_1(0)t+o(t),\,\,\textnormal{ as } t\to 0.$$
After the iteration in the corresponding equations, we obtain that
\begin{equation*}
x_i(t)=\frac{(\prod_{j=2}^{i}a_{j,j-1}(0))\cdot x_1(0)\cdot t^{i-1}}{(i-1)!}+o(t^{i-1}),\,\,\textnormal{ as } t\to 0,
\end{equation*}
for each $2\leq i\leq n$.
Since $a_{i,i-1}(0)>0$ for $2\leq i\leq n$, it is clear that $x_i(t)$ shares the same symbol with $x_1(0)$, and hence,
$x(t)\in \Lambda$ for all $t>0$ small enough. This implies that $N(x(t))=1$ for $t>0$ small enough. Since $N(x)\geq 1$ for any $x\in\Lambda$, we have $N(x(t))=N_m(x(0))$ for $t>0$ small. On the other hand, for $t<0$ with $\abs{t}$ small enough, the symbol of $x_i(t)$ will change
alternately with respect to the index $i$. As a consequence, $N(x(t))=N_M(x(0))=\tilde{n}$, for $t<0$ with $\abs{t}$ small. So, we have proved this lemma for the special case of $x(0)=(x_1,0,\cdots,0)$.
By repeating the argument above, one can obtain this lemma for the case of $x(0)=(0,\cdots,x_j,\cdots,0)$ with $x_j\neq 0$.

We now consider the general case. Given any index $j\ge 1$ with $x_j(0)=0$ and any index $i$ with $x_i(0)\ne 0$, one can follow the same argument as in the paragraphs above to obtain that, for $\abs{t}>0$ small, the symbol of $x_j(t)$ can be determined by the
symbol of the sum
\begin{equation}\label{sum-x-j-ex}
\left[\sum_{\substack{1\le i<j \\x_i(0)\ne 0}}(\prod_{k=i}^{j-1}a_{k+1,k}(0))\cdot x_i(0)\cdot t^{j-i}\right]+c_j\cdot\left[\sum_{\substack{j< i\le n \\x_i(0)\ne 0}}(\prod_{k=i}^{n-1}a_{k+1,k}(0))\cdot x_i(0)\cdot t^{n+j-i}\right],
\end{equation}
where $c_j=a_{1,n}(0)\cdot\prod_{l=1}^{j-1}a_{l+1,l}(0)$. Here, we set $\prod_{k=n}^{n-1}a_{k+1,k}(0)=1.$

Based on \eqref{sum-x-j-ex}, one can define an index set
$J:=\{j:x_j(0)=0\}$. Note that $x(0)\ne 0$ and $x(0)\notin \mathcal{N}$. Then $J$ is a nonempty proper subset of
$\{1,\cdots,n\}$. Now we partition $J$ into a finite union of pairwise disjoint integer segments $J_1,\cdots, J_m$ (mod $n$, e.g., $J_1=\{n-1,n,1,2\}$). For each $J_s$, one may write $J_s=\{j_s,j_s+1,\cdots, j_s+n_s\}$ (indices mod $n$), and then $j_s-1,j_s+n_s+1\notin J$.

We first consider the case (i): $1\notin J$. In this case, for any $J_s$ and any $j\in J_s$, it follows from \eqref{sum-x-j-ex} that the symbol ${\rm sgn}[x_j(t)]$ of $x_j(t)$ (choose $\abs{t}>0$ smaller, if necessary) is determined by $(\prod_{k=j_s-1}^{j-1}a_{k+1,k}(0))\cdot x_{j_s-1}(0)\cdot t^{j-j_s+1}$. Note that $1\notin J_s$. Then $\prod_{k=j_s-1}^{j-1}a_{k+1,k}(0)>0$, which implies that ${\rm sgn}[x_j(t)]$ is determined by $x_{j_s-1}(0)\cdot t^{j-j_s+1}$.
As a consequence, if $t>0$ is sufficiently small, then ${\rm sgn}[x_j(t)]={\rm sgn}[x_{j_s-1}(0)]$ for any $j\in J_s$. This entails that, for $t>0$ small enough, $J_s$ contributes no increase for $N$ in the neighborhood of $x(0)$. Due to arbitrariness of $J_s$, it then follows that $N(x(t))=N_m(x(0))$ for $t>0$ sufficiently small. On the other hand, if $t<0$ is sufficiently small, then ${\rm sgn}[x_j(t)]=(-1)^{j-j_s+1}\cdot{\rm sgn}[x_{j_s-1}(0)]$ for any $j\in J_s$. Therefore, for $t<0$ small enough,  $J_s$ contributes the largest increase for $N$ in the neighborhood of $x(0)$. So $N(x(t))=N_M(x(0))$ for $t<0$ sufficiently small.
Thus, \eqref{min-max} has been verified in this case.

Now consider the case (ii): $1\in J$. In this case, there is a unique $s_*$ such that $1\in J_{s_*}$. For any
integer segment $J_s$ of $J$ satisfying $1\notin J_s$, one can repeat the same argument in the previous paragraph and obtain that, for $t>0$ small enough, $J_s$ contributes no increase for $N$ in the neighborhood of $x(0)$; and for $t<0$ small enough,  $J_s$ contributes the largest increase for $N$ in the neighborhood of $x(0)$.

So, it suffices to consider $J_{s_*}$. We write $$J_{s_*}=\{j_{s_*},j_{s_*}+1,j_{s_*}+2,\cdots,1,\cdots, j_{s_*}+n_{s_*}\} \,\,\textnormal{(indices mod} \,n\textnormal{)}.$$ Following such notation, we define a subset $R\subset J_{s_*}$ as $R=\{j\in J_{s_*}|j=1,\text{ or }j\text{ is on ``the right side" of }1\}$. If $j\in R$ then $j<j_{s_*}-1\le n$. Together with $j_{s_*}-1\notin J$, it then follows from \eqref{sum-x-j-ex} that ${\rm sgn}[x_j(t)]$ is determined by the symbol of $c_j\cdot(\prod_{k=j_{s_*}-1}^{n-1}a_{k+1,k}(0))\cdot x_{j_{s_*}-1}(0)\cdot t^{n+j-(j_{s_*}-1)}$
whenever $\abs{t}$ is sufficiently small. Then \eqref{positive-a-a} implies that,
for any $\abs{t}$ sufficiently small,
\begin{eqnarray}\label{111}
{\rm sgn}[x_j(t)]=\left\{
\begin{split}
 -{\rm sgn}[x_{j_{s_*}-1}(0)],\,\text{ if }j\in R\text{ and }t>0;\\
 (-1)^{n+j-j_{s_*}}{\rm sgn}[x_{j_{s_*}-1}(0)],\,\text{ if }j\in R\text{ and }t<0,
\end{split}\right.
\end{eqnarray}
If $j\in J_{s_*}\setminus R$, then $1<j_{s_*}-1<j$. Again by \eqref{sum-x-j-ex}, we obtain that ${\rm sgn}[x_j(t)]$ is determined by the symbol of $(\prod_{k=j_{s_*}-1}^{j-1}a_{k+1,k}(0))\cdot x_{j_{s_*}-1}(0)\cdot t^{j-j_{s_*}+1}$
whenever $\abs{t}$ is small. Thus,
\begin{eqnarray}\label{112}
{\rm sgn}[x_j(t)]=\left\{
\begin{split}
 {\rm sgn}[x_{j_{s_*}-1}(0)],\,\text{ if }j\in J_{s_*}\setminus R\text{ and }t>0;\\
 (-1)^{j-j_{s_*}+1}{\rm sgn}[x_{j_{s_*}-1}(0)],\,\text{ if }j\in J_{s_*}\setminus R\text{ and }t<0;
\end{split}\right.
\end{eqnarray}
for any $\abs{t}$ sufficiently small.

Therefore, if $t>0$ is sufficiently small, then ${\rm sgn}[x_j(t)]={\rm sgn}[x_{j_{s_*}-1}(0)]$ for $j\in J_{s_*}\setminus R$, and ${\rm sgn}[x_j(t)]=-{\rm sgn}[x_{j_{s_*}-1}(0)]$ for $j\in R$. Noticing that $\delta_1=-1$ and $\delta_i=1 (2\leq i\leq n)$ in the definition of $N$,  one obtains that, for $t>0$ sufficiently small,
$J_{s_*}$ contributes no increase for $N$ in the neighborhood of $x(0)$. Similarly, by virtue of the expression of ${\rm sgn}[x_j(t)]$ in \eqref{111}-\eqref{112}, $J_{s_*}$ contributes the largest increase for $N$ in the neighborhood of $x(0)$, for $t<0$ is sufficiently small.

As a consequence, for case (ii), we have also obtained that $N(x(t))=N_m(x(0))$ for $t>0$ sufficiently small and  $N(x(t))=N_M(x(0))$ for $t<0$ sufficiently small. Thus, we have completed the proof.
\end{proof}

Motivated by  \cite{Fusco1987}, for any given integer $0\leq h\leq \frac{\tilde{n}+1}{2}$, let $K_h$ and $K^h$ be the sets
\begin{eqnarray*}
\begin{split}
 K_h=\{0\}\cup\{x\in \mathbb{R}^n: N_M(x)\leq 2h-1\},\\
 K^h=\{0\}\cup\{x\in \mathbb{R}^n:N_m(x)> 2h-1\}.
\end{split}
\end{eqnarray*}
In particular, we set $K_0=\{0\}$ and $K^{0}=\mathbb{R}^n$. It is not difficult to see that $K_h\setminus \{0\}$ and $K^h\setminus \{0\}$ are open sets, $K_h \cap K^h=\{0\}$ and the closure $\overline{K_h\cup K^h}=\mathbb{R}^{n}$.

Hereafter, we denote by $\bar{K}_h$ (resp. $\bar{K}^h$) the closure of $K_h$ (resp. $K^h$), by ${\rm Int}\bar{K}_h$ the interior of $\bar{K}_h$. Since $0\in \overline{(K_h\setminus \{0\})}$, $\bar{K}_h=\overline{(K_h\setminus \{0\})}$. Recall that $K_h\setminus \{0\}$ is an open set, we have ${\rm Int}\bar{K}_h=K_h\setminus \{0\}$.
\begin{proposition}\label{map-int}
Let $\Phi(t)$ be a fundamental matrix of \eqref{linearity-com-system} with $\Phi(0)=I$. Then for any $t>0$, one has
\begin{equation*}
  \Phi(t)(\bar{K}_h\setminus\{0\})\subset {\rm Int}\bar{K}_h.
\end{equation*}
\end{proposition}
\begin{proof}
Suppose that there exist $x_0\in \bar{K}_h\setminus\{0\}$ and $t_0>0$, such that $\Phi(t_0)x_0\notin  K_h\setminus \{0\}$. Then $N_M(\Phi(t_0)x_0)>2h-1$. Since $\mathcal{N}$ is open and dense, one can find a sequence $x_n\in \mathcal{N} \cap (K_h\setminus \{0\})$ (which entails that $N(x_n)\leq 2h-1$) such that $x_n\rightarrow x_0$ as $n\rightarrow \infty$. On the other hand, by Lemma \ref{zero-property2}, one can choose $\epsilon_0>0$ small enough, such that $t_0-\epsilon_0>0$ and $N(\Phi(t_0-\epsilon_0)x_0)=N_M(\Phi(t_0)x_0)>2h-1$. Since $\mathcal{N}$ is an open set and $N(\cdot)$ is continuous on $\mathcal{N}$, one has $\Phi(t_0-\epsilon_0)x_n\in \mathcal{N}$, and hence, $N(\Phi(t_0-\epsilon_0)x_n)=N(\Phi(t_0-\epsilon_0)x_0)>2h-1$ for $n$ sufficiently large, which contradicts the fact that $N(\Phi(t_0-\epsilon_0)x_n)\leq N(x_n)\leq 2h-1$. We have completed the proof.
\end{proof}
Based on Proposition \ref{map-int}, we give the following corollary which is useful in the forthcoming section.

\begin{Cor}\label{solution-space}
Let $A(t)$ be the coefficient matrix of \eqref{linearity-com-system}. Then:
\begin{itemize}
\item [\emph{(i)}]If $\Sigma_{0}\subset K_h$ is a linear subspace and $\Sigma_{t}$ is the image of $\Sigma_{0}$ at time $t$ under \eqref{linearity-com-system},
    then ${\rm dim}\Sigma_t={\rm dim}\Sigma_0$ and $\Sigma_t\subset K_h$ for all $t\geq 0$.
\item [\emph{(ii)}]If $\Sigma^{0}\subset K^h$ is a linear subspace and $\Sigma^{t}$ is the image of $\Sigma^{0}$ under \eqref{linearity-com-system}, then ${\rm dim}\Sigma^t={\rm dim}\Sigma^0$ and $\Sigma^t\subset K^h$ for all $t\leq 0$.
\end{itemize}
\end{Cor}
\begin{proof}
We only prove (i), the other case is similar. It is easy to see that ${\rm dim}\Sigma_t=\text{dim } \Sigma_0$ for all $t\in \mathbb{R}$, by the standard  solution theory of homogeneous linear differential equations. For any nonzero vector $x_0\in \Sigma_0$, by Proposition \ref{map-int}, $\Phi(t)x_0\in{\rm Int}\bar{K}_h=K_h\setminus\{0\}$ for all $t>0$, where $\Phi(t)$ is the solution operator of \eqref{linearity-com-system}. So $\Sigma_t\subset K_h$ for all $t\geq 0$.
\end{proof}

We now introduce the concept of a cone of rank $k$ (see \cite{Kra,fusco1991,San1}):
\begin{definition}
{\rm
Let $k\ge 1$ be an integer. A closed subset $K\subset \mathbb{R}^n$ is called {\it a cone of rank $k$},
if for any $x\in K$ and $\lambda \in \mathbb{R}$, one has $\lambda x\in K$. Moreover,
$\max\{\mathrm{dim} W|W\text{ is a subspace of }\mathbb{R}^n\text{ and } W\subset K \}=k$.}
\end{definition}
\begin{rem}
{\rm It is easy to see that a usual convex cone $C$ (in the sense of Hirsch \cite{HiSm})
defines the cone $K=C\cup (-C)$ which is of rank $1$.}
\end{rem}

\begin{proposition}\label{dim-cone}
For each $h=1,\cdots,\frac{\tilde{n}-1}{2},$ $\bar{K}_h$ is a cone of rank $2h$. More precisely, let $V$ be a subspace of $\mathbb{R}^n$. Then
\begin{equation*}
  d_h=\mathrm{max} \{\mathrm{dim} V|V\subset\bar{K}_h\}=2h.
\end{equation*}
\end{proposition}
Before proving this proposition, we need a technical lemma as follows.
\begin{lemma}\label{space-split}
Let $A$ be a $n\times n$ matrix of the following form
\begin{equation*}
A=    \left(
      \begin{array}{ccccc}
        0 & 0 &  &  -1 \\
        1 & 0 & 0 &    \\
        & \ddots & \ddots & 0\\
        0 &  &   1 & 0 \\
      \end{array}
    \right).
\end{equation*}
Then one has:
\begin{itemize}
 \item [\emph{(i)}]If $n$ is even, there exist $\frac{\tilde{n}+1}{2}$ invariant subspaces $E_k$ of $A$, with $\mathrm{dim}E_k=2$ for $k=1,\cdots,\frac{\tilde{n}+1}{2}$. Moreover, for any nonzero vector $\xi\in E_k$, one has $\xi\in \mathcal{N}$ and $N(\xi)=2k-1$.
 \item [\emph{(ii)}]If $n$ is odd, there exist $\frac{\tilde{n}+1}{2}$ invariant subspaces $E_k$ of $A$, with $\mathrm{dim}E_k=2$, for $k=1,\cdots,\frac{\tilde{n}-1}{2}$, and $\mathrm{dim}E_{\frac{\tilde{n}+1}{2}}=1$. Moreover, for any nonzero vector $\xi\in E_k$, one has $\xi\in \mathcal{N}$ and $N(\xi)=2k-1$.
 \item [\emph{(iii)}]Let $W_{i,j}=E_i\oplus \cdots\oplus E_j$, for $1\leq i\leq j\leq \frac{\tilde{n}+1}{2}$. Then for any nonzero vector $\xi\in W_{i,j}$, one has
\begin{equation*}
  2i-1\leq N_m(\xi)\leq N_M(\xi)\leq 2j-1.
\end{equation*}
\end{itemize}
\end{lemma}

\begin{proof}
We only prove (i), because the proof of (ii) is similar. Since the characteristic polynomial of this matrix $A$ is $\lambda^n+1$, the eigenvalues of this matrix are $\lambda_k=\cos\frac{2(k-1)\pi+\pi}{n}+i\sin\frac{2(k-1)\pi+\pi}{n}$, $k=1,\cdots,n$ and the corresponding eigenvectors are $\eta_k=(\lambda^{n-1}_k,\lambda^{n-2}_k,\cdots,1)^T$, $k=1,\cdots,n$.
Because $n$ is even, all the roots are conjugate complex roots.

Let $E_k={\rm span}\{{\rm Re}\eta_k,{\rm Im}\eta_k\}$ for $k=1,2,\cdots,\frac{\tilde{n}+1}{2}$, then these spaces are invariant under $A$. Moreover, dim$E_k=2$ for $k=1,2,\cdots,\frac{\tilde{n}+1}{2}$. Clearly, ${\rm Re}\eta_k$ and ${\rm Im}\eta_k$ belong to $\mathcal{N}$ and $N({\rm Re}\eta_k)=N({\rm Im}\eta_k)=2k-1$ for $k=1,2,\cdots,\frac{\tilde{n}+1}{2}$. Given any $\xi\in E_k\setminus\{0\}$, the solution $x(t)$ of $\dot{x}=Ax$ with initial value $x(0)=\xi$ can be expressed as
\begin{equation*}\label{repre-solution}
  x(t)=(c_kq_k(t)+\tilde{c}_k\tilde{q}_k(t))e^{\mu_k t},
\end{equation*}
where $\mu_k={\rm Re} \lambda_k$, $q_k(t)$ and $\tilde{q}_k(t)$ are periodic functions with $q_k(0)={\rm Re}\eta_k$ and $\tilde{q}_k(0)={\rm Im}\eta_k$. By Lemma \ref{zero-property}(iii), there exists $T_0>0$ and $l,s\in \mathbb{N}$ such that $N(x(t))=l$ for $t>T_0$ and $N(x(t))=s$ for $t<-T_0$. Since $(c_kq_k(t)+\tilde{c}_k\tilde{q}_k(t))$ is also periodic function,  we have $s=l$. Consequently, Lemma \ref{zero-property}(ii) implies that $x(t)\in\mathcal{N}$ and $N(x(t))=l$ for all $t\in \mathbb{R}$. In particular, $\xi\in \mathcal{N}$. By the arbitrariness of $\xi$, we have $E_k\setminus\{0\}\subset \mathcal{N}$. Recall that $N({\rm Re}\eta_k)=N({\rm Im}\eta_k)=2k-1$. Combining with the connectivity of $E_k\setminus\{0\}$, the continuity of $N$ on $\mathcal{N}$ then implies that $N(\xi)=2k-1$ for all $\xi\in E_k\setminus\{0\}$.

For (iii), we also consider the case that $n$ is even, the other case is similar. Choose a nonzero vector $\xi\in W_{i,j}$, then $\xi={\Sigma}_{k=i}^j(c_k{\rm Re}\eta_k+\tilde{c}_k{\rm Im}\eta_k)$. Without loss of generality, we assume that $c_k\neq 0$ and $\tilde{c}_k\neq 0$, for $k=i,\cdots,j$. Similar as in (i), the solution $x(t)$ of $\dot{x}=Ax$ with initial value $x(0)=\xi$, can be represented in the following form
\begin{equation*}
  x(t)={\Sigma}_{k=i}^j(c_kq_k(t)+\tilde{c}_k\tilde{q}_k(t))e^{\mu_k t},
\end{equation*}
where $\mu_k={\rm Re}\lambda_k$, $q_k(t)$ and $\tilde{q}_k(t)$ are periodic functions with $q_k(0)={\rm Re}\eta_k$ and $\tilde{q}_k(0)={\rm Im}\eta_k$, for $k=i,\cdots,j$. Moreover, we note that $\mu_i>\cdots>\mu_j$.\par
 From Lemma \ref{zero-property}(iii), there exist $T_0>0$ and $l,s\in \mathbb{N}$ such that $N(x(t))=l$ for all $t\geq T_0$ and $N(x(t))=h$ for all $t\leq-T_0$. Since $q_k(t)$ and $\tilde{q}_k(t)$  are periodic for $k=i,\cdots,j$, there exist two sequences $t_m\rightarrow -\infty$ and $\tilde{t}_m\rightarrow \infty$ as $m\rightarrow \infty$ such that $e^{-\mu_jt_m}x(t_m)\rightarrow (c_j{\rm Re}\eta_j+\tilde{c}_j{\rm Im}\eta_j)$ and $e^{-\mu_i\tilde{t}_m}x(\tilde{t}_m)\rightarrow (c_i{\rm Re}\eta_i+\tilde{c}_i{\rm Im}\eta_i)$ as $m\rightarrow\infty$. By virtue of (i) of this lemma, it entails that $h=2j-1$ and $l=2i-1$. So, $2i-1=N(x(T_0))\leq N_m(x(0))\leq N_M(x(0))\leq N(x(-T_0))= 2j-1$. We have completed the proof.
\end{proof}

\begin{proof}[Proof of Proposition \ref{dim-cone}]
It is easy to see that $d_h\geq 2h$ from Lemma \ref{space-split} (iii), by choosing $i=1,j=h$. Suppose that $d_h>2h$, then there exists a subspace $V_1\subset \bar{K}_h$ with dim $V_1>2h$. Thus, one can choose at least $(2h+1)$ linearly independent column-vectors $\xi_1,\cdots,\xi_{2h+1}\in V_1$. For $y=\Sigma_{i=1}^{2h+1}\gamma_i\xi_i$, since $B=(\xi_1,\cdots,\xi_{2h+1})$ is an $n\times {(2h+1)}$ matrix with ${\rm Rank}(B)=2h+1$, by choosing $\gamma_i$  suitably, we may obtain some $y$ whose $2h+1$ components are equal to $1$ or $-1$, alternatively. This then implies $N_m(y)\geq 2h+1$. On the other hand, since $y\in \bar{K}_k$ and the open and dense of $\mathcal{N}$, there exists a sequence $x_n\in \bar{K}_h\cap \mathcal{N}$ such that, $N(x_n)\leq 2h-1$ and $x_n\rightarrow y$ as $n\rightarrow \infty$, which means $N_m(y)\leq 2h-1$, a contradiction. Thus, we have proved that $d_h=2h$.
\end{proof}

\begin{rem}\label{K-cone}
{\rm
By virtue of Proposition \ref{dim-cone}, we obtain that $\bar{K}_h$ (resp. $\bar{K}^h$), for $h=1,\cdots, \frac{\tilde{n}-1}{2}$, are cones with rank $\bar{K}_h=2h$ (resp. rank $\bar{K}^h=n-2h$).
}
\end{rem}

In order to generalize the Floquet Theory in \cite{Mallet-Paret1990} for time-periodic cases to general time-dependent cases, we need the following generalized Perron-Frobenius Theorem (See e.g., \cite[Theorem 1]{fusco1991}).
\begin{lemma}\label{perron-thm}
Let $K\subset \mathbb{R}^n$ be a cone of rank $d$. Assume that $L$ is a linear operator on $\mathbb{R}^n$ satisfying $L(K\setminus\{0\})\subset {\rm Int} K$.
Then there exist (unique) subspaces $V_1$, $V_2$ such that
\begin{itemize}
\item [\emph{(i)}] $V_1\cap V_2=\{0\}$, $\mathrm{dim}$ $V_1=d$, $\mathrm{dim}$ $V_2=n-d$,
\item [\emph{(ii)}] $LV_j\subset V_j$, $j=1,2$,
\item [\emph{(iii)}] $V_1\subset \{0\}\cup {\rm Int} K$, $V_2\cap K=\{0\}$.
\end{itemize}
Moreover, if $\sigma_1(L)$ and $\sigma_2(L)$ are the spectra of $L$ restricted to $V_1$ and $V_2$, then between $\sigma_1(L)$ and $\sigma_2(L)$ there is a gap:
$$\lambda\in \sigma_1(L),\ \mu\in\sigma_2(L)\Rightarrow |\lambda|>|\mu|.$$
\end{lemma}

Now we are ready to present the following proposition which generalizes the Floquet Theory obtained in \cite{Mallet-Paret1990} for time-periodic cases.
\begin{proposition}\label{rootspace-arra}
Let $\Phi(t)$ be a fundamental matrix of \eqref{linearity-com-system} with $\Phi(0)=I$. Then for any fixed $t>0$, there exist subspaces $W_h$, $h=1,2,\cdots,\frac{\tilde{n}+1}{2}$, which are invariant with respect to $\Phi(t)$ and satisfy:
\begin{equation*}
\begin{aligned}
   {\rm dim} W_h=2,\quad h=1,\cdots,\frac{\tilde{n}-1}{2},\\
   {\rm dim} W_{\frac{\tilde{n}+1}{2}}=\left\{
   \begin{aligned}
   & 2 \quad \text{if}\ n=\tilde{n}+1\ \text{is even},  \\
   & 1 \quad \text{if}\ n=\tilde{n}\ \text{is odd},  \\
   \end{aligned}
  \right.
\end{aligned}
\end{equation*}
and
$$\mathbb{R}^n=W_1\oplus W_2\oplus\cdots \oplus W_{\frac{\tilde{n}+1}{2}}.$$
If $x\in W_i\setminus\{0\}$ then $x\in \mathcal{N}$ and $N(x)=2i-1$, for $i=1,\cdots,\frac{\tilde{n}+1}{2}$. If
$x\neq 0$ and $ x\in W_h\oplus W_{h+1}\cdots \oplus W_k$, then $N_m(x)\geq 2h-1$ and $N_M(x)\leq 2k-1$.
Moreover, if $\nu_i$ and $\mu_i$ are the minimum and the maximum module of characteristic values of the restriction of $\Phi(t)$ to $W_i$, then
\[
\mu_1\geq\nu_1>\mu_2\geq\nu_2>\cdots>\mu_{\frac{\tilde{n}+1}{2}}\geq \nu_{\frac{\tilde{n}+1}{2}}.
\]
\end{proposition}
\begin{proof}
For any fixed $t>0$, It follows from Lemma \ref{map-int} and Remark \ref{K-cone} that $\Phi(t)$ and $\bar{K}_h$ satisfy the assumptions in Lemma \ref{perron-thm}. As a consequence, if we let $d_h=\text{max}\{\text{dim}V|V\ \text{a subspace},\ V\subset \bar{K}_h\}$, then there exist subspaces $V_h^1$, $V_h^2$ which are invariant under $\Phi(t)$ satisfying ${\rm dim}V_h^1=d_h$, ${\rm dim}V_h^2=n-d_h$, $\mathbb{R}^n=V_h^1\oplus V_h^2$, $V_h^1\subset \bar{K}_h$ and $V_h^2\cap\bar{K}_h=\{0\}$. Moreover, if $\sigma_h^1$ and $\sigma_h^2$ are the spectra of the restriction of $\Phi(t)$ to $V_h^1$ and $V_h^2$, then for any $\lambda^1\in \sigma_h^1$ and $\lambda^2\in \sigma_h^2$, one has $|\lambda^1|>|\lambda^2|$.

 Since $K_1\subset K_2\subset\cdots\subset K_{\frac{\tilde{n}+1}{2}}$, we have $V_1^1\subset V_2^1\subset\cdots\subset V_{\frac{\tilde{n}+1}{2}}^1$ and $V_1^2 \supset V_2^2\supset \cdots\supset V_{\frac{\tilde{n}+1}{2}}^2$.  Let $W_h=V_h^1\cap V_{h-1}^2$, for $h=1,\cdots,\frac{\tilde{n}+1}{2}$ (here $V_0^2=\mathbb{R}^n$). Then it is clear that all these $W_h$'s are invariant under $\Phi(t)$.
 Moreover,
 \[
 \text{dim} W_h=d_h-d_{h-1};\quad W_h\cap\bar{K}_{h-1}=\{0\}.
 \]
By Lemma \ref{dim-cone}, $d_h=2h$ for $h=1,\cdots,\frac{\tilde{n}-1}{2}$. Then it yields that
$\text{dim} W_h=2$ for $h=1,\cdots \frac{\tilde{n}-1}{2},$ and $\text{dim} W_{\frac{\tilde{n}+1}{2}}=2\textnormal{ or }\,1$ (for $n\text{ being even or odd} $), and
$$\mathbb{R}^n=W_1\oplus W_2\oplus\cdots\oplus W_{\frac{\tilde{n}+1}{2}}.$$
Note also that $W_h\oplus\cdots\oplus W_k\subset V_k^1\cap V_{h-1}^2$, $V_k^1\subset K_k$ and $V_{h-1}^2\cap\bar{K}_{h-1}=\{0\}$. Then for any nonzero vector $x\in W_h \oplus\cdots\oplus W_k$, one has $x\in K_{k}$ but $x\notin \bar{K}_{h-1}$ (here $\bar{K}_{0}=\{0\}$). So, $N_M(x)\leq 2k-1$ and $N_m(x)\geq 2h-1$. In particular, for $h=k$, $N_M(x)=N_m(x)=2h-1$. Finally, the fact that $W_h\subset V_h^1$, $W_{h+1}\subset V_h^2$ and $|\lambda^1|>|\lambda^2|$ whenever $\lambda^1\in \sigma_h^1$ and $\lambda^2\in \sigma_h^2$ implies that, if $\nu_h$ and $\mu_h$ are the minimum and maximum module of characteristic values of $\Phi(t)|W_h$, then $\nu_h>\mu_{h+1}$ for $h=1,\cdots,\frac{\tilde{n}-1}{2}$. Thus, we have proved the lemma.
\end{proof}

\section{Transversality}
In this section, we will prove that stable and unstable manifolds of two hyperbolic periodic solutions (or a hyperbolic equilibrium and a hyperbolic periodic orbit) of \eqref{feedback-equation} intersect transversely. Furthermore, we will point out that, under certain condition, stable and unstable manifolds of two hyperbolic equilibriums also intersect transversely.

Before we proceed our approach, it deserves to point out that a change of variables $x_i\rightarrow \mu_{i}x_i$, where $\mu_i \in \{-1,1\}$ are appropriately chosen, yields a MCFS \eqref{feedback-equation} with negative feedback, where $\delta_1=-1$ and $\delta_i=1,2\leq i \leq n$. Hereafter, we always assume that $\delta_1=-1$ and $\delta_i=1,2\leq i \leq n$.

Let $p(t)$ be an $\omega$-periodic solution of \eqref{feedback-equation} with $\omega>0$ and $\Gamma$ be the orbit of $p(t)$. Consider the linearized equation of \eqref{feedback-equation} along $p(t)$:
\begin{equation*}
  \dot{z}=Df(p(t))z,\quad t\in \mathbb{R},\quad z\in \mathbb{R}^n,
\end{equation*}
which is an $\omega$-periodic linear equation in the form of \eqref{linearity-com-system}. $p(t)$ is called {\it hyperbolic} if none of its Floquet multipliers is on the unit-circle $\mathbb{S}^1\subset \mathbb{C}$ except $1$.\par
Let $A\subset \mathbb{R}^n$ be a nonempty subset of $\mathbb{R}^n$, the distance from a point $x_0\in \mathbb{R}^n$ to $A$ is defined by $d(x_0,A)=\underset{x\in A}{{\rm inf}}\norm{x_0-x}.$
We write $\phi(t,x)$ as the solution of \eqref{feedback-equation}-\eqref{feedback-condition} satisfying $\varphi(0,x)=x$. Now define the {\it stable {\rm (resp}. unstable{\rm )} manifold} $W^s(\Gamma)$ (${\rm resp}.\ W^u(\Gamma)$) of $\Gamma$ as
\begin{equation*}
\begin{split}
&W^s(\Gamma)=\{x\in \mathbb{R}^n| \ \lim\limits_{t \to +\infty}{d(\varphi(t,x),\Gamma)}=0 \},\\
&W^u(\Gamma)=\{x\in \mathbb{R}^n| \ \lim\limits_{t \to +\infty}{d(\varphi(-t,x),\Gamma)}=0 \}.
\end{split}
\end{equation*}
It is known that $W^s(\Gamma)$ and $W^u(\Gamma)$ are $C^1$-manifolds (See \cite[Chapter 1]{Chicon}). Two smooth submanifolds $M$ and $N$ of $\mathbb{R}^n$ are said to {\it intersect transversely} (written as $M\pitchfork N$) if either $M\cap N=\emptyset$ or at each point $x\in M\cap N$, the tangent spaces $T_xM$, $T_xN$ span $\mathbb{R}^n$. For briefly, we write $\varphi(t)=\varphi(t,x)$.

Our main result in this section is the following
\begin{theorem}\label{transversality}
Let $\varphi(t)$ be a solution of \eqref{feedback-equation} which connects two hyperbolic periodic orbits $\Gamma^{-}$ and $\Gamma^{+}$. Then
\begin{equation*}
W^u(\Gamma^{-})\pitchfork W^s(\Gamma^{+}).
\end{equation*}
\end{theorem}
The proof of Theorem \ref{transversality} can be broken into Proposition \ref{transversal1} and Proposition \ref{transversal2}.  Before proving these propositions, we give some notations and useful lemmas.

 Hereafter, we let $Q=\{x|x=\varphi(t),\ t\in (-\infty,+\infty)\}$ with the initial value $\varphi(0)=x_0$ and the two hyperbolic periodic orbits $\Gamma^{\pm}=\{x|x=p^{\pm}(t),\ t\in [0,\omega^{\pm})\}$, where $\omega^{\pm}>0$ is the minimum positive period of $p^{\pm}(t)$ in Theorem \ref{transversality}.  Since $\Gamma^{\pm}$ is hyperbolic, there exists a tubular neighborhood of $\Gamma^{\pm}$ and a $C^1$-fibration $\mathcal{F}^{\pm}$ (see e.g., \cite{Hirsch}), which is positively (or negatively) invariant under the flow of \eqref{feedback-equation}. The existence of such foliation implies that $p^{\pm}(t)$ can be chosen so that:
\begin{equation}\label{limit-approach}
\lim\limits_{t \to \pm \infty}{\|\varphi(t)-p^{\pm}(t)\|}=0.
\end{equation}

\begin{lemma}\label{linear-variation}
There exist $t_0>0$ and $h^+,h^-\in \{1,2,\cdots,\frac{\tilde{n}+1}{2}\}$, satisfying $h^+\le h^-$,  such that
\begin{equation*}
N(\dot{\varphi}(t))=2h^+-1,\quad \text{for } t\geq t_0;\quad N(\dot{\varphi}(t))=2h^--1,\quad \text{for } t\leq -t_0.
\end{equation*}
Moreover, $N(\dot{p}^+(t))=2h^+-1$ and $N(\dot{p}^-(t))=2h^--1$, for all $t\in \mathbb{R}$.
\end{lemma}
\begin{proof}
Clearly, $\dot{\varphi}(t)$ is a solution of the linear equation $\dot{y}=Df(\varphi(t))y$. Note that $Df(\varphi(t))$ is a coefficient matrix of type \eqref{linearity-com-system}, then the existence of $h^+$, $h^-$ and $t_0$ is confirmed by Lemma \ref{zero-property}(iii).

Because $\dot{p}^{\pm}(t)$ is a periodic solution of $\dot{y}=Df(p^{\pm}(t))y$, $N(\dot{p}^{\pm}(t))$ is well defined for all $t\in \mathbb{R}^n$ and independent of $t$. By \eqref{limit-approach}, we may assume that $\Gamma^{\pm}$ and $Q$ are in a compact set $M\subset \mathbb{R}^n$. Together with the uniform continuity of $f(x)$ on $M$,  \eqref{limit-approach} implies that
\begin{equation*}
  \lim\limits_{t \to \pm \infty}{\|\dot{\varphi}(t)-\dot{p}^{\pm}(t)\|}=0.
\end{equation*}
It then follows from continuity of $N$ that $N(\dot{p}^{\pm}(t))=2h^{\pm}-1$.
\end{proof}
Henceforth, we let $\Phi^{\pm}(t,s)$ and $\Phi(t,s)$ $(t\ge s)$, be the solution operator of the linear equations $\dot{y}=Df(p^{\pm}(t))y$ and $\dot{y}=Df(\varphi(t))y$, respectively. For briefly, we write $\Phi^{\pm}=\Phi^{\pm}(\omega^{\pm},0)$. Then $\dot{p}^{\pm}(0)$ is an eigenvector of $\Phi^{\pm}$ corresponding to the (simple) eigenvalue $1$. By virtue of Lemma \ref{rootspace-arra}, one may define the module of characteristic values of $\Phi^{\pm}$ by
\[
\mu^{\pm}_1\geq\nu^{\pm}_1>\mu^{\pm}_2\geq\nu^{\pm}_2>\cdots>\mu^{\pm}_{\frac{\tilde{n}+1}{2}}\geq \nu^{\pm}_{\frac{\tilde{n}+1}{2}},
\]
and hence,
\begin{equation}\label{spectrum}
\nu^-_{h^--1}>1,\quad \mu^+_{h^++1}<1,
\end{equation}
where $h^+$ and $h^-$ are defined in Lemma \ref{linear-variation} satisfying $h^+\leq h^-$. In the following, we will consider the case of (i) $h^+<h^-$; (ii) $h^+=h^-$, respectively.\par
\begin{proposition}\label{transversal1}
Let $\Gamma^+,\ \Gamma^-,$ and $\varphi$ be defined as in Theorem \ref{transversality}. Then if $h^+<h^-$, one has
 \[
 W^u(\Gamma^-)\pitchfork W^s(\Gamma^+).
 \]
\end{proposition}
\begin{proof}
 By Lemma \ref{rootspace-arra}, if we let $\Sigma^-$ be the eigenspace of $\Phi^-$ defined as $\Sigma^-=W^-_1\oplus\cdots \oplus W^-_{h^--1}$ , then \eqref{spectrum} implies that $\Sigma^-\subset T_{p^-(0)}W^u(\Gamma^-)$. Moreover, $\Sigma^-\subset K_{h^--1}$ and dim $ \Sigma^-=2h^--2$. Note that $W^u(\Gamma^-)$ is a smooth manifold and $K_{h^--1}\setminus\{0\}$ is an open set. Then, for the positive integer $j\in \mathbb{Z}$ sufficiently large, there is a $2h^--2$ dimensional subspace $\tilde{\Sigma}^-\subset T_{\varphi(-j\omega^-)}W^u(\Gamma^-)\cap K_{h^--1}$. Now let $\tilde{\Sigma}^-_0$ be the image of $\tilde{\Sigma}^-$ under $\Phi(0,-j\omega^-)$, then $\tilde{\Sigma}^-_0$ is a linear subspace of $\mathbb{R}^n$. It then follows from  Corollary \ref{solution-space} (i) that
\begin{equation*}
 \text{dim } \tilde{\Sigma}^-_0=2h^--2 \,\, \text{ and } \,\, \tilde{\Sigma}^-_0\subset T_{\varphi(0)}W^u(\Gamma^-)\cap K_{h^--1}.
\end{equation*}
On the other hand, let $\Sigma^+=W^+_{h^-}\oplus\cdots\oplus W^+_{\frac{\tilde{n}+1}{2}}$ be the eigenspace of $\Phi^+$. Then, by ${\mu}^+_{h^++1}<1$ and $h^+<h^-$, we have $\Sigma^+ \subset T_{p^+(0)}W^s(\Gamma^+)$. Similarly as above, one can find a subspace $\tilde{\Sigma}^+_0$ of $\mathbb{R}^n$ such that
\[
\tilde{\Sigma}^+_0 \subset T_{\varphi(0)}W^s(\Gamma^+)\cap K^{h^--1}\,\, \text{ with } \,\, \text{dim }\tilde{\Sigma}^+_0=n-2h^-+2.
\]
Recall that $K_{h^--1}\cap K^{h^--1}=\{0\}$, then $\tilde{\Sigma}^+_0\cap \tilde{\Sigma}^-_0=\{0\}$. Combining the fact that $\text{dim }\tilde{\Sigma}^+_0+\text{dim }\tilde{\Sigma}^-_0=n$, we obtain that $\tilde{\Sigma}^+_0\oplus \tilde{\Sigma}^-_0=\mathbb{R}^n$. This completes the proof.
\end{proof}

Now we consider the second case, that is, $h^+=h^-$. Motivated by \cite{fusco1990}, we first give the following technical lemma.
\begin{lemma}\label{zero-constant}
Let $\Omega=\Gamma^{+}\cup Q\cup\Gamma^{-}$, if $h^+=h^-=h$ in Lemma \ref{linear-variation}, then one has
\[
y-x\in \mathcal{N} \quad{\rm and}\ N(y-x)=2h-1
\]
for any distinct $x,y\in \Omega$.
\end{lemma}
\begin{proof}
Now choose $x,y\in \Omega$ with $x\neq y$. We consider the following three cases.\par
Case (i). If $x,\ y\in \Gamma^+$, then by the definition of $\Gamma^+$ there exist $r,\ s\in[0,\omega^+)$ with $r\neq s$, such that $x=p^+(r)$, $y=p^+(s)$. Let $q^+(t)=p^+(s+t)-p^+(r+t)$, then $q^+(t)$ is a periodic function and it satisfies the linear system \eqref{linearity-com-system} with
\[
a_{ij}(t)=\int_{0}^{1}\frac{\partial f_i}{\partial x_j}(u_{i-1}(l,t),u_i(l,t))dl,
\]
where $u_j(l,t)=lp_j^+(s+t)+(1-l)p_j^+(r+t)$, $j=i-1,i$. Here we write $a_{10}(t)=a_{1n}(t)$ and $x_0=x_n$. So $q^+(t)\in \mathcal{N}$ for all $t\in \mathbb{R}$, in particular, one has $q^+(0)=y-x\in \mathcal{N}$. Since $\Gamma^+\times \Gamma^+$ is homeomorphic to $S^1\times S^1$, let $\Delta=\{(x,x)|x\in \Gamma^+\}$, then $(\Gamma^+\times \Gamma^+)\setminus \Delta$ is a connected set. Note that the map $(\Gamma^+\times \Gamma^+)\setminus \Delta\rightarrow\mathbb{R}^n;
(x,y)\rightarrow y-x$ is continuous, then $M^+=\{y-x|x,y\in\Gamma^+,y\neq x\}$ is also connected. By the continuity of $N$ and connectivity of $M^+$, $N$ is a constant on $M^+$. Note also that
\[
y-x=p^+(s)-p^+(r)=(s-r)\dot{p}^+(r)+o(s-r),
\]
for $|s-r|$ (hence $||y-x||$) sufficiently small, one has
$N(y-x)=N(\dot{p}^+(r))$ if $\norm{y-x}$ is sufficiently small. Hence, by Lemma \ref{linear-variation}, $N(y-x)=2h-1$ for such $x$ and $y$, which implies that $N=2h-1$ on $M^+$.\par
For $x,y\in \Gamma^-$. The same argument yields that $M^-=\{y-x|x,y\in\Gamma^-,y\neq x\}\subset\mathcal{N}$ and $N=2h-1$ on $M^-$.\par
Case (ii). If $x,y\in Q$, then there exist $r,s\in (-\infty,+\infty)$ with $r\neq s$, such that $x=\varphi(r)$ and $ y=\varphi(s)$. Let $q(t)=\varphi(s+t)-\varphi(r+t)$, then it follows from \eqref{limit-approach} that
\begin{equation}\label{approach-differnece}
 \lim\limits_{t \to \pm \infty}{\|q(t)-q^{\pm}(t)\|}=0,
\end{equation}
where $q^{\pm}(t)=p^{\pm}(s+t)-p^{\pm}(r+t)$.

If $|r-s|$ is not a multiple of $\omega^+$ or $\omega^-$, then by Lemma \ref{zero-property}(iii), one has $q(t)\in \mathcal{N}$ for $|t|$ large enough. Moreover, by case (i) one has already known that $N(q^{\pm}(t))=2h-1$. So, \eqref{approach-differnece} implies that $N(q(t))=2h-1$ for all $t\in \mathbb{R}$. In particular, $N(y-x)=N(q(0))=2h-1$.

If $|r-s|=k\omega^+$ for some positive integer $k$, we also claim that $q(0)=y-x\in \mathcal{N}$. For otherwise, it follows from Lemma \ref{zero-property} (ii) that $N(q(-\varepsilon))>N(q(\varepsilon))$ for all $\varepsilon>0$ small, and hence,
\begin{equation}\label{q+-q-com}
\textnormal{ either } N(q(-\varepsilon))\neq 2h-1, \textnormal{ or } N(q(\varepsilon))\neq 2h-1.
\end{equation}
On the other hand, one can choose sequences $q^{\pm}_k=\varphi(s\pm \varepsilon+\frac{1}{k})-\varphi(r\pm \varepsilon)$. By the statement in the previous paragraph, one obtains $q^{\pm}_k\in \mathcal{N}$ and $N(q^{\pm}_k)=2h-1$ for $k$ sufficiently large and $q^{\pm}_k\rightarrow q(\pm\varepsilon)$ as $k\rightarrow \infty$, contradicting \eqref{q+-q-com}. Thus, $q(0)\in \mathcal{N}$. Moreover, choose $\tilde{q}_k=\varphi(s+\frac{1}{k})-\varphi(r), k=1,2,\cdots$, again we obtain $\tilde{q}_k\rightarrow q(0)$ and $N(\tilde{q}_k)=2h-1$. Consequently, $N(y-x)=N(q(0))=2h-1$.

Case (iii). For general $x,y\in \Omega$. If $y-x\in \mathcal{N}$, one can choose sequences $y_n,x_n\in Q$ approaching $y$ and $x$. So, by case (ii), $N(y-x)=N(y_n-x_n)=2h-1$. If $y-x\notin \mathcal{N}$, then there always exist $\bar{x},\bar{y}\in \Omega$ with $\|\bar{x}-x\|$ and $\|\bar{y}-y\|$ sufficiently small such that $\bar{y}-\bar{x}\in \mathcal{N}$ and $N(\bar{y}-\bar{x})\neq 2h-1$, which contradicts that $N(\bar{y}-\bar{x})= 2h-1$. We have proved this lemma.
\end{proof}

Now we finish the proof of Theorem \ref{transversality} by proving the following Proposition \ref{transversal2}.
\begin{proposition}\label{transversal2}
Let $\Gamma^+,\ \Gamma^-,$ and $\varphi$ be defined as in Theorem \ref{transversality}. If $h^+=h^-=h$, then
 \[
 W^u(\Gamma^-)\pitchfork W^s(\Gamma^+).
 \]
\end{proposition}
\begin{proof}
 Choose a subsequence $\{t_k\}\subset \{-l\omega^-\}_{l=1}^\infty$ and let $w^k=\frac{\varphi(t_k)-p^-(0)}{||\varphi(t_k)-p^-(0)||}$ for $k=1,\cdots,n\}$. Without loss of generality, we may assume that
$w^k$ converges to $w$ as $k\to \infty$. Now, we write
$a^{(k)}_{ij}(t)=\int_{0}^{1}\frac{\partial f_i}{\partial x_j}(u^{(k)}_{i-1}(l,t),u^{(k)}_i(l,t))dl$, where $u_j^{(k)}(l,t)=l\varphi_j(t+t_k)+(1-l)p_j^-(t)$, $j=i-1,i$. By \eqref{limit-approach}, $\varphi(t)$ and $p^{\pm}(t)$ are bounded and uniformly continuous on $\mathbb{R}$, and hence, the sequence of matrix-valued functions $A^{(k)}(t)=(a^{(k)}_{ij}(t))$ is equicontinuous and uniformly bounded. By the Ascoli-Arzela's lemma, there is a subsequence of $A^{(k)}(t)$, still denoted by $A^{(k)}(t)$, which converges to $Df(p^-(t))$ uniformly for $t$ on  any compact interval.

Let $\phi^{(k)}(t)=\frac{\varphi(t+t_k)-p^-(t)}{\|\varphi(t_k)-p^-(0)\|}$. Then, by a standard result in the theory of ordinary differential equations \cite[Lemma 3.1,Chapter I]{Hale}, $\phi^{(*)}(t)$ is a solution of $\dot{z}=Df(p^-(t))z$ with $\phi^{(*)}(0)=w$ and $\phi^{(*)}(t)=\lim\limits_{k \to  \infty}{\phi^{(k)}(t)}$, uniformly for $t$ on any compact interval.
We claim that $\phi^{(*)}(t)\in \mathcal{N}$ and $N(\phi^{(*)}(t))=2h-1$ for all $t\in \mathbb{R}$. Indeed, by Lemma \ref{zero-property}(iii), one can find a $t_0>0$ such that $\phi^{(*)}(t)\in \mathcal{N}$ and $N(\phi^{(*)}(t))=N_1$ (resp. $N_2$) for all $t\geq t_0$ (resp. $t\leq -t_0$). Fix such $t_0$, it follows from the continuity of $N$ that $N(\phi^{(k)}(t_0))=N_1$ and $N(\phi^{(k)}(-t_0))=N_2$ for all $k$ sufficiently large. By virtue of Lemma \ref{zero-constant}, we obtain that $N_1=N_2=2h-1$, and hence, it entails that $\phi^{(*)}(t)\in \mathcal{N}$ and $N(\phi^{(*)}(t))=2h-1$ for all $t\in \mathbb{R}$. Thus we have proved the claim. \par

Noticing that $w^k=\phi^{(k)}(0)$, one has $w=\phi^{(*)}(0)$. Hence, the claim implies that $N(w)=2h-1$. Thus, by Lemma \ref{rootspace-arra}, we obtain that $w\in W^-_{h}$.
 Since $\varphi(t_k)\rightarrow p^-(0)$ and $\varphi(t_k)\in \mathcal{F}^-_{p^-(0)}$ for $k$ sufficiently large, $w$ is tangent to the fiber $\mathcal{F}^-_{p^-(0)}$ at $p^-(0)$. So, $w$ is linearly independent of $\dot{p}^-(0)$, and hence, $W^-_{h}=\text{span}\{w,\ \dot{p}^-(0)\}$. Moreover, noticing that $\varphi(t_k), p^-(0)\in W^u(\Gamma^-)$. Then $w\in T_{p^-(0)}W^u(\Gamma^-)$ and $T_{p^-(0)}W^u(\Gamma^-)\supseteq W^-_{1}\oplus\cdots\oplus W^-_{h}$. On the other hand, observing that $T_{p^-(0)}W^u(\Gamma^-)\subseteq W^-_{1}\oplus\cdots\oplus W^-_{h}$. Then, one has $T_{p^-(0)}W^u(\Gamma^-)= W^-_{1}\oplus\cdots\oplus W^-_{h}$.

  Now let $\Sigma^-=W^-_1\oplus\cdots \oplus W^-_{h}$ and $\Sigma^+=W^+_{h+1}\oplus\cdots \oplus W^+_{\frac{\tilde{n}+1}{2}}$. Recall that $\Sigma^+ \subset T_{p^+(0)}W^s(\Gamma^+)$. Then, similarly as the argument  in Proposition \ref{transversal1}, one can obtain the transversality, which complete our proof.
\end{proof}

 Now we will consider the case that there is an orbit $\varphi(t)$ connecting between a hyperbolic equilibrium and a hyperbolic periodic orbit or two hyperbolic equilibria.
  An equilibrium $e$ of \eqref{linearity-com-system} is called {\it hyperbolic} if $Df(e)$ does not possess any eigenvalue whose real part is equal to $0$. Denote by $W^s(e)$ and $W^u(e)$ the {\it stable} and {\it unstable} manifold of $e$, respectively. Then, we have:
\begin{lemma}\label{hyperbolic-fixed}
  Let $\varphi(t)$ be a solution of \eqref{feedback-equation}. Assume that $\varphi(t)$ connects two hyperbolic critical points $e^+,\ e^-$, then we have:
  \[
  {\rm dim}W^u(e^+)\leq {\rm dim}W^u(e^-).
  \]
  In particular, if ${\rm dim}W^u(e^+)< {\rm dim}W^u(e^-)$, then
  $  W^s(e^+)\pitchfork W^u(e^-).$
\end{lemma}
\begin{proof}
By Lemma \ref{zero-property}(iii), there exist $h^+,\ h^-\in \{1,2,\cdots,\frac{\tilde{n}+1}{2}\}$ with $h^+\leq h^-$, and some $t_0>0$, such that $N(\dot{\varphi}(t))=2h^+-1$ (resp. $N(\dot{\varphi}(t))=2h^--1$) for all $t\ge t_0$ (resp. $t\le -t_0$).

  It follows from Proposition \ref{rootspace-arra} that there are $\frac{\tilde{n}+1}{2}$ invariant spaces $W^+_i,\ i=1,\cdots,\frac{\tilde{n}+1}{2}$ of the matrix $\exp\{Df(e^+)\}$ and the module of the corresponding eigenvalues satisfied $\mu_1^+\geq\nu_1^+>\mu_2^+\geq\nu_2^+>\cdots>\mu_{\frac{\tilde{n}+1}{2}}^+\geq \nu_{\frac{\tilde{n}+1}{2}}^+$. Let $m$ be the minimum integer that $\nu_m^+<1$, then
\[
T_{e^+}W^s(e^+)\subset W^+_{m}\oplus\cdots\oplus W^+_{\frac{\tilde{n}+1}{2}}\subset K^{m}.
\]
Clearly, $\lim\limits_{t \to  \infty}{\varphi(t)}=e^+$ implies that $\dot{\varphi}(t)\in T_{\varphi(t)}W^s(e^+)$. Since $W^s(e^+)$ is a $C^1$ manifold and $K^{m}\setminus\{0\}$ is an open set, one obtains that $\dot{\varphi}(t)\in T_{\varphi(t)}W^s(e^+)\subset K^{m}$ for all $t>0$ sufficiently large.

Recall that $N(\dot{\varphi}(t))=2h^+-1$ for all $t\ge t_0$. Then $h^+\ge m$. Note also that
 ${\rm dim} W^u(e^+)\leq 2m-1$. It follows that ${\rm dim}W^u(e^+)\leq 2h^+-1$. A similar argument with respect to $e^-$ yields that  ${\rm dim} W^u(e^-)\geq 2h^--1$. As a consequence, $${\rm dim}W^u(e^+)\leq 2h^+-1\le 2h^--1\le {\rm dim}W^u(e^-).$$
Let $m^+=\text{dim}W^u(e^+)$ and $m^-=\text{dim}W^u(e^-)$. If $m^+< m^-$, then one can replace $h^-$ and $h^+$ by $[\frac{m^-+1}{2}]$ and $[\frac{m^+}{2}]$ in Proposition \ref{transversal1}. Note that $h^+<h^-$ in this case. Then
 one can repeat the proof in Proposition \ref{transversal1} to obtain that $W^s(e^+)\pitchfork W^u(e^-)$.
\end{proof}

\begin{theorem}\label{general-trans}
Let $\varphi(t)$ be a solution of \eqref{feedback-equation}. Assume that $\varphi(t)$ connect two hyperbolic critical elements $\gamma^+,\ \gamma^-$(fixed point or periodic orbit), then:
  \[
  W^u(\gamma^-)\pitchfork W^s(\gamma^+)
  \]
  provided one of the following condition holds:
 \item [\emph{(i)}]One of these two hyperbolic critical elements is a periodic orbit and the other is fixed point.
 \item [\emph{(ii)}] $\gamma^+$ and $\gamma^-$ are fixed points, moreover {\rm dim} $W^u(\gamma^+)<{\rm dim}W^u(\gamma^-)$.
\end{theorem}
\begin{proof}
For (i), without loss of generality, we assume that $\gamma^+$ is an equilibrium and denote it by $e^+$. Then, from Lemma \ref{hyperbolic-fixed}, we have dim$W^s(e^+)\geq n-2h^++1$ which means that $W^+_{h^++1}\oplus\cdots\oplus W^+_{\frac{\tilde{n}+1}{2}}\subset T_{e^+}W^s(e^+)$ . Let $\Sigma^-=W^-_{1}\oplus\cdots\oplus W^-_{h^+}$ and $\Sigma^+=W^+_{h^++1}\oplus\cdots\oplus W^+_{\frac{\tilde{n}+1}{2}}$, then similarly as in Proposition \ref{transversal1}, we have $W^u(\gamma^-)\pitchfork W^s(\gamma^+)$.
For (ii), see Lemma \ref{hyperbolic-fixed}.
\end{proof}

\end{document}